\documentclass[11pt]{article}
\usepackage{amsmath,amsthm,amsfonts,amssymb,amscd, amsxtra}

\newcommand{\banacha}{\mathbb X}
\newcommand{\banachb}{\mathbb Y}
\newtheorem{theorem}{Theorem}
\newtheorem{lemma}[theorem]{Lemma}

\newtheorem{corollary}[theorem]{Corollary}
\newtheorem{proposition}[theorem]{Proposition}

\newtheorem{remark}{Remark}
\newtheorem{example}{Example}
\begin{document}
\title{Local convergence of the Gauss-Newton method for injective-overdetermined systems of equations under a majorant condition}

\author{ M.L.N. Gon\c calves \thanks{IME/UFG, Campus II- Caixa
    Postal 131, CEP 74001-970 - Goi\^ania, GO, Brazil (E-mail:{\tt
      maxlng@mat.ufg.br}).  The author was partly supported by
    CNPq Grant 473756/2009-9 and CAPES.} }
 \maketitle
\begin{abstract}

A local convergence analysis of the Gauss-Newton method for
solving  injective-overdetermined  systems of nonlinear equations
under a majorant condition is provided. The convergence  as well as results on
its rate are established  without a convexity hypothesis on the
derivative of the majorant function. The optimal convergence
radius,  the biggest range for  uniqueness of the solution along
with  some other special cases are also obtained.

\end{abstract}

\noindent {{\bf Keywords:} Injective-overdetermined systems of equations;
Gauss-Newton method; Majorant condition; Local convergence.}

\maketitle
\section{Introduction}\label{sec:int}
Let $\banacha$ and $\banachb$ be real or complex Hilbert spaces. Let $\Omega\subseteq\banacha$ be an open set, and
 $F:\Omega\to \banachb$ a continuously differentiable nonlinear function. Consider the {\it systems of nonlinear equations }
\begin{equation}\label{eq:11}
F(x)=0.
\end{equation}
If $F'(x)$ is invertible, the Newton method and its variants (see
\cite{F08,F10,MAX1,FS06}) are the most efficient methods known for
solving such systems. However, $F'(x)$ may not even  be  a square
matrix. One simple example arises when $\banacha=\mathbb{R}^n$ and
$\banachb=\mathbb{R}^m$, with $n\neq m$. In this case, $F'(x)$ is
not  invertible and \eqref{eq:11} becomes an overdetermined system
$(n<m)$ or an underdetermined system $(n>m)$. In general, if $F'(x)$
is injective or surjective, we say \eqref{eq:11} is an
injective-overdetermined  or surjective-underdetermined system  of
equations, respectively.

If $F '(x)$ is not necessarily invertible, a generalized Newton
method called the Gauss-Newton method  can be used (see
\cite{MAX2,MAX3}). It is  defined by
$$
x_{k+1}={x_k}- F'(x_k)^{\dagger}F(x_k),
\qquad k=0,1,\ldots,
$$
where $F'(x_k)^{\dagger}$ denotes the Moore-Penrose inverse of the
linear operator $F'(x_k)$. This algorithm finds least-squares
solutions of \eqref{eq:11}. These least-squares solutions, which may
or may not be solutions of the original problem \eqref{eq:11}, are
related to the nonlinear least squares problem
$$
\min_{x\in \Omega } \;\|F(x)\|^2,
$$
that is, they are stationary points of $H(x)=\|F(x)\|^2$. This paper
is focused on the case in which the least-squares solutions of
\eqref{eq:11}   also solve \eqref{eq:11}. In the theory of nonlinear
least squares problems, this case is called the zero-residual case.

Regarding the local and semi-local convergence analysis of the Newton
and Gauss-Newton methods, in the last years there has been
 much work  attempting to alleviate  the assumption of Lipschitz continuity on the operator $F'$,
 see for example \cite{AAA,C11,F08,F10,MAX1,MAX2,MAX3,FS06,huangy2004,LZJ,S86,XW10,XW9,10100,1010}.
 The main conditions that relax the Lipschitz continuity on the derivative is the  majorant condition,
used for example in \cite{F08,F10, MAX1,MAX2,MAX3,FS06}, and the
generalized Lipschitz condition according to X.Wang, used for example in
\cite{AAA,C11,huangy2004, LZJ,XW10,XW9,10100,1010}.  In fact, as
proved in \cite{F10}, if the majorant function has convex derivative,
 these conditions are equivalent. Otherwise, the  Wang's
condition can be seen as a particular case of the majorant
condition. Moreover,  the majorant formulation provides a clear
relationship between the majorant function and the nonlinear
operator under consideration,  simplifying  the proof of convergence
substantially.

Our aim in this paper is to present a new local convergence analysis of the Gauss-Newton method for solving  injective-overdetermined systems of equations under a majorant condition. The convergence, uniqueness, superlinear rate
 and an estimate of the best possible convergence radius will be established without a convexity hypothesis on the derivative of the majorant function, which was assumed in \cite{MAX2}. In addition to the special cases obtained in~\cite{MAX2}, the lack of  convexity of the derivative of the majorant function in this analysis, allows us to obtain two new important special cases, namely, the convergence can be ensured under H\"{o}lder-like  and  generalized Lipschitz conditions.
  In the latter case,  the results are obtained without assuming that the function that defines the condition is nondecreasing, thus generalizing Corollary~6.3 in \cite{C11}.
 Moreover, it is worth to mention that, similarly to the convergence analysis of the  Newton method (see \cite{F10}),  the  hypothesis of convex derivative of the majorant function or nondecreasing of the
function  which defines  the generalized Lipschitz condition, are needed only to obtain quadratic  convergence rate.

The organization of the paper is as follows. In Section
\ref{sec:int.1}, we list some notations and one basic result used in our presentation.  In Section \ref{lkant}, we state the main result and in
Section~\ref{sec:PR} some properties of the
majorant function are established and the main relationships between the majorant function and the
nonlinear function $F$ are presented. The optimal ball of convergence and the uniqueness of the solution  are also
discussed in Section ~\ref{sec:PR}. In Section \ref{sec:proof} our main result is proven and
some applications of this result are obtained in Section \ref{apl}. Some final remarks are offered in Section~\ref{fr}
\subsection{Notation and auxiliary results} \label{sec:int.1}
The following notations and results are used throughout our
presentation.   Let $\banacha$ and $\banachb$ be Hilbert spaces. The open and closed ball
at $a \in \banacha$ with radius $\delta>0$ are denoted, respectively by
$$
B(a,\delta) :=\{ x\in \banacha ;\; \|x-a\|<\delta \}, \qquad B[a,\delta] :=\{ x\in \banacha ;\; \|x-a\|\leqslant \delta \}.
$$
The set $\Omega\subseteq\banacha$ is an open set, the function $F:\Omega\to \banachb$ is continuously differentiable, and $F'(x)$  has a closed image in $\Omega$.

Some properties related to the Moore-Penrose inverse will be needed.
More details about the Moore-Penrose inverse can be found in \cite{G1,W1}.

Let $A: \banacha \to \banachb$ be a continuous and injective linear operator with closed image. The Moore-Penrose inverse $A^\dagger:\banachb \to \banacha$ of $A$ is defined by
$$
A^\dagger:=(A^*A)^{-1} A^*,
$$
where  $A^*$ denotes the adjoint of the linear operator $A$.

\begin{lemma} \label{lem:ban2}
Let $A, B: \banacha \to \banachb$ be a continuous linear operator with closed image. If $A$ is injective and
$\|A^\dagger \|\|A-B\|<1$, then $B$ is injective and $$
\|B^\dagger\|\leq \frac{\|A^\dagger\|}{ 1- \|A^\dagger\|\|A-B\|}.
$$
\end{lemma}
\section{Local analysis  for the Gauss-Newton method } \label{lkant}
Our goal is to state and prove a local theorem for the Gauss-Newton method, which generalizes the Corollary~8 of \cite{MAX2},
as well as Theorem~2 of \cite{F10}.  First, we  prove some results regarding the scalar
majorant function, which relaxes the Lipschitz condition. Then, we establish the main relationships between the
majorant function and the nonlinear function $F$. We also obtain the optimal ball of convergence and  the uniqueness
of the solution in a suitable region. Finally, we show well definedness and convergence,
along with results on the convergence rates.  The statement of the theorem~is:

\begin{theorem}\label{th:nt}
Let $\banacha$ and $\banachb$ be  Hilbert spaces,
 $\Omega\subseteq \banacha$ be an open set and
$F:{\Omega}\to \banachb$ be a continuously differentiable
function such that $F'$  has a closed image in $\Omega$. Let $x_* \in \Omega,$ $R>0$,
$
\beta:=\|F'(x_*)^{\dagger}\| $ and $ \kappa:=\sup \left\{ t\in [0, R): B(x_*, t)\subset\Omega \right\}.
$
Suppose that $F(x_*)=0$,
$F '(x_*)$ is injective
and there exists a $f:[0,\; R)\to \mathbb{R}$
continuously differentiable such that
  \begin{equation}\label{Hyp:MH}
\beta\left\|F'(x)-F'(x_*+\tau(x-x_*))\right\| \leq
f'\left(\|x-x_*\|\right)-f'\left(\tau\|x-x_*\|\right),
  \end{equation}
  for  all $\tau \in [0,1]$,  $x\in B(x_*, \kappa)$  and
\begin{itemize}
  \item[{\bf h1)}]  $f(0)=0$ and $f'(0)=-1$;
  \item[{\bf  h2)}]  $f'$ is   strictly increasing.
\end{itemize}
Let $\nu  :=\sup \left\{t \in[0, R):f'(t)<0\right\},$
$ \rho :=\sup \left\{\delta \in(0, \nu):{[f(t)/f'(t)-t]}/{t}<1,\; t \;\in \;(0, \delta)\right\}$ and
  $$  r :=\min  \left\{\kappa, \, \rho \right\}
$$
Then, the sequences $\{x_k\}$ and $\{t_k\}$, with starting points $x_0\in B(x_*, r)/\{x_*\}$ and $t_0=\|x_0-x_*\|$, respectively, such that
\begin{equation} \label{eq:DNS}
   x_{k+1}={x_k}-F'(x_k)^{\dagger}F(x_k), \qquad
   t_{k+1} =|t_k-f(t_k)/f'(t_k)|, \qquad k=0,1,\ldots\,,
\end{equation}
are well defined; $\{t_k\}$  is strictly decreasing,  contained in $(0, r)$ and it converges to $0$. Furthermore, $\{x_k\}$ is contained
in $B(x_*,r)$, it converges to the point $x_*$, which is the unique zero of $F$ in $B(x_*, \sigma)$, where
$ \sigma:=\sup\{t\in(0, \kappa): f(t)< 0\}$, and there hold:
\begin{equation}\label{eq:q2}
    \lim_{k \to \infty} \; [{\|x_{k+1}-x_*\|}\big{/}{\|x_k-x_*\|}]=0,\qquad \lim_{k \to \infty}[{t_{k+1}}\big{/}{t_k}]=0.
\end{equation}
Moreover, if $ f(\rho)/(\rho f'(\rho))-1=1$
and $\rho<\kappa$, then $r=\rho$ is the best possible convergence
radius.\\
If, additionally, given $0\leq p\leq1$
\begin{itemize}
  \item[{\bf  h3)}]  the function  $(0,\, \nu) \ni t \mapsto [f(t)/f'(t)-t]/t^{p+1}$ is  strictly increasing,
\end{itemize}
   then the sequence $\{t_{k+1}/t_k^{p+1}\}$
 is strictly decreasing  and we have:
\begin{equation}\label{eq:q3}
\|x_{k+1}-x_*\| \leq \big[t_{k+1}/t_k^{p+1}\big]\,\|x_k-x_*\|^{p+1}\leq \big[t_{1}/t_0^{p+1}\big]\,\|x_k-x_*\|^{p+1}, \qquad k=0,1,\ldots\,.
  \end{equation}
  Consequently, for $k\geq0$,
$$ \|x_{k}-x_*\| \leq\left\{
                             \begin{array}{ll}
                               t_0[t_1/t_0]^k, & \hbox{if \quad p=0;} \\
                               t_0[t_1/t_0]^{((p+1)^k-1)/p}, & \hbox{if \quad p$\neq$0.}
                             \end{array}
                           \right.$$

\end{theorem}

\begin{remark}
If $F '(x_*)$ is invertible in Theorem~\ref{th:nt} we obtain  the
local convergence of the Newton method for  systems of nonlinear
equations, as obtained in Theorem 2 of \cite{F10}.
\end{remark}

\begin{remark}
In particular, if $f'$ is convex,  we can prove that {\bf h3} holds
with $p=1$ and, therefore, in this case, we are led to the result
proven in Corollary~8  of \cite{MAX2}, i.e.,  the local convergence
of the Gauss-Newton method for solving injective-overdetermined systems
of equations.   Hence, the additional assumption that the majorant
function, $f$,  has convex derivative, is only necessary  in order to
obtain quadratic convergence rate. This behavior is similar for the
Newton method, see  \cite{F10}.
\end{remark}

\begin{example}(see \cite{F10})
The following  continuously differentiable functions satisfy {\bf h1}, {\bf h2} and  {\bf h3}:
\begin{itemize}
  \item[{i)}] $f: [0, +\infty)\to \mathbb{R}$ such that $f(t)=t^{1+p}-t $;
  \item[{ii)}] $f: [0, +\infty)\to \mathbb{R}$ such that $f(t)=\mbox{e}^{-t}+t^2-1$.
\end{itemize}
If $0<p<1$,  the derivatives of both  functions  are not convex.
\end{example}

From now on, we assume that all the assumptions of Theorem \ref{th:nt}
hold, with the exception of {\bf h3},  which will be considered to hold only when explicitly stated.
\subsection{Preliminary results} \label{sec:PR}
In this section, we will prove all the statements in Theorem~\ref{th:nt} regarding  the sequence $\{t_k\}$ associated to the  majorant function. The main relationships between the  majorant function  and the nonlinear operator will be also established, as well as the results in Theorem~\ref{th:nt}  related to the uniqueness of the solution and the optimal convergence radius.
\subsubsection{The scalar sequence} \label{sec:PMF}
In this part,  we will check the statements in Theorem~\ref{th:nt}
involving  $\{t_k\}$.

First of all, it is easy to see that the hypothesis  {\bf h1},   {\bf h2} and {\bf
h3} in Theorem~\ref{th:nt} coincide with those one used in Theorem~2 of
\cite{F10}. Moreover,  the constants $ \kappa,\, \nu $, $\rho$ and
$\sigma$ also coincide. Hence, the proofs in this section, which can
be found in section~2.1.1 of \cite{F10},  will be omitted.

\begin{proposition}  \label{pr:incr1}
The constants $ \kappa,\, \nu $ and $\sigma$ are positive and $t-f(t)/f'(t)<0,$ for all $t\in (0,\,\nu).$
\end{proposition}
\begin{proof}
The proof follows  as the one of Proposition~3 of \cite{F10}.
\end{proof}
According to {\bf h2} and the definition of $\nu$, we have  $f'(t)< 0$ for all
$t\in[0, \,\nu)$.  Therefore, the Newton iteration map for  $f$ is well defined in
$[0,\, \nu)$. Let us call it $n_f$:
\begin{equation} \label{eq:def.nf}
  \begin{array}{rcl}
  n_f:[0,\, \nu)&\to& (-\infty, \, 0]\\
    t&\mapsto& t-f(t)/f'(t).
  \end{array}
\end{equation}
\begin{proposition}  \label{pr:incr102}
$
\lim_{t\to 0}|n_f(t)|/t=0.
$
As a consequence,  $\rho>0 $ and
$|n_f(t)|<t$ for all $ t\in (0, \, \rho)$.
\end{proposition}
\begin{proof}
See the proof of Proposition~4 of \cite{F10}.
\end{proof}

Using \eqref{eq:def.nf}, it is easy to see that  the sequence $\{t_k \}$ is equivalently defined as
\begin{equation} \label{eq:tknk}
 t_0=\|x_0-x_*\|, \qquad t_{k+1}=|n_f(t_k)|, \qquad k=0,1,\ldots\, .
\end{equation}
\begin{corollary} \label{cr:kanttk}
The sequence $\{t_k\}$ is well defined, strictly decreasing and  contained in $(0, \rho)$. Moreover,  $\{t_k\}$ converges to $0$ with superlinear rate, i.e.,
$
\lim_{k\to \infty}t_{k+1}/t_k=0.
$
If, additionally, {\bf  h3} holds,  the sequence $\{t_{k+1}/t_k^{p+1}\}$ is strictly decreasing.
\end{corollary}
\begin{proof}
The proof follows the same ideas of the proof of Corollary~5 of
\cite{F10}.
\end{proof}

\subsubsection{Relationship of the majorant function with the nonlinear function} \label{sec:MFNLO}
In this part we present the main relationships between the majorant
function, $f$, and the nonlinear function, $F$.
\begin{lemma} \label{wdns}
If \,$\| x-x_*\|<\min\{\nu,\kappa\}$, then
$F'(x)^* F'(x) $ is invertible and
$$
\left\|F'(x)^{\dagger}\right\|\leq {\beta}/{|f'(\|
x-x_*\|)|}.
 $$
In particular, $F'(x)^* F'(x)$ is invertible in $B(x_*, r)$.
\end{lemma}
\begin{proof}
As $\| x-x_*\|<\min\{\nu,\kappa\}$, we have $f'(\| x-x_*\|)<0$. Hence,  using the definition of $\beta$, the inequality \eqref{Hyp:MH} and {\bf h1},  we have
\begin{equation}\label{eq1010}
\|F'(x_*)^{\dagger}\|\|F'(x)-F'(x_*)\|=\beta\|F'(x)-F'(x_*)\|\leq
f'(\| x-x_*\|)-f'(0)< 1.
\end{equation}
Since $F'(x_*)$  is injective, \eqref{eq1010} implies, in view
of Lemma \ref{lem:ban2}, that  $F'(x)$ is injective. So, $F'(x)^*
F'(x) $ is invertible and, by the definition of
 $r$, we obtain that $F'(x)^* F'(x)$ is invertible for all $x\in B(x_*, r)$.
Moreover, from Lemma \ref{lem:ban2} we also have
$$
\left\|F'(x)^{\dagger}\right\|\leq \frac{\beta}{1-\beta\|F'(x)-F'(x_*)\| }\leq \frac{\beta}{1-(f'(\| x-x_*\|)-f'(0))}=\frac{\beta}{|f'(\| x-x_*\|)|} ,
 $$
where $f'(0)=-1$ and $f'<0$ in $[0,\nu)$ are used for obtaining the
last equality.
\end{proof}

Now, it is convenient to study the linearization error of $F$ at
point in~$\Omega$. For this we define
\begin{equation}\label{eq:def.er}
  E_F(x,y):= F(y)-\left[ F(x)+F'(x)(y-x)\right],\qquad y,\, x\in \Omega.
\end{equation}
We will bound this error by the error in the linearization of the
majorant function $f$
\begin{equation}\label{eq:def.erf}
        e_f(t,u):= f(u)-\left[ f(t)+f'(t)(u-t)\right],\qquad t,\,u \in [0,R).
\end{equation}
\begin{lemma}  \label{pr:taylor}
If  $\|x-x_*\|< \kappa$, then  $\beta \|E_F(x, x_*)\|\leq
e_f(\|x-x_*\|, 0).$
\end{lemma}
\begin{proof}
 Since   $B(x_*, \kappa)$ is convex,  we obtain that $x_*+\tau(x-x_*)\in B(x_*, \kappa)$, for $0\leq \tau \leq 1$.
 Thus,  as $F$ is  continuously differentiable in $\Omega$, the definition of $E_F$ and some simple manipulations yield
$$
\beta\|E_F(x,x_*)\|\leq  \int_0 ^1 \beta \left\|
[F'(x)-F'(x_*+\tau(x-x_*))]\right\|\,\left\|x_*-x\right\| \;
d\tau.
$$
From  the last inequality  and  assumption \eqref{Hyp:MH}, we obtain
$$
\beta\|E_F(x,x_*)\| \leq \int_0 ^1
\left[f'\left(\left\|x-x_*\right\|\right)-f'\left(\tau\|x-x_*\|\right)\right]\|x-x_*\|\;d\tau.
$$
Evaluating the above integral and using the definition of $e_f$, the
statement follows.

\end{proof}
In particular, Lemma \ref{wdns} guarantees that  $F'(x)^*F'(x)$
is invertible in $B(x_*, r)$ and, consequently, the Gauss-Newton
iteration map is well defined.  Let $G_{F}$ be, the
Gauss-Newton iteration map for $F$ in that region:
\begin{equation} \label{NF}
  \begin{array}{rcl}
  G_{F}:B(x_*, r) &\to& \banachb\\
    x&\mapsto& x- F'(x)^{\dagger}F(x).
  \end{array}
\end{equation}
In the next proposition, we will establish an important relationship
between the maps $n_{f}$ and $ G_{F}$. Consequently, we obtain that
$B(x_*, r)$ is invariant under  $ G_{F}$. This result will be
very  important to ensure the good definition of the
Gauss-Newton method.
\begin{lemma} \label{l:wdef}
If \,$\| x-x_*\|<r$,  then $\|G_F(x)-x_{*}\|\leq |n_f(\| x-x_*\|)|.$
Consequently,
$$G_{F}(B(x_*, r))\subset B(x_*, r).$$
\end{lemma}
\begin{proof}
The first inequality is trivial for $x=x_*$, since $F'(x_*)^{\dagger}F(x_*)=0$. Now, assume that
 $0<\|x-x_*\|<~r$. Lemma \ref{wdns} implies that $F'(x)^* F'(x)$ is invertible. Hence, using  $F(x_*)=0$, some
algebraic manipulation and \eqref{NF},  the following holds
$$
G_F(x)-x_{*}=F'(x)^{\dagger}[F'(x)(x-x_*)-F(x)+F(x_*)].
$$
From the last inequality, \eqref{eq:def.er} and Lemmas
\ref{wdns} and \ref{pr:taylor}, we obtain
\[
\|G_F(x)-x_{*}\| \leq \|F'(x)^{\dagger}\|\|E_{F}(x,x_*)\| \leq
\beta\|E_{F}(x,x_*)\|/|f'(\|x-x_*\|)|\leq
 e_f(\|x-x_*\|, 0)/|f'(\|x-x_*\|)|.
\]
On the other hand, taking into account that $f(0)=0,$ the
definitions of $e_f$ and $n_f$ imply that
$$ e_f(\|x-x_*\|, 0)/|f'(\|x-x_*\|)|=|n_f(\| x-x_*\|)|.$$
Hence, the first statement follows by combining the last two
inequalities.

For the second assertion, take $x\in B(x_*, r)$. Since
$0<\|x-x_*\|<r\leq \rho$, the first inequality of the lemma and the last
inequality of Proposition~\ref{pr:incr102}  imply that $
\|G_F(x)-x_*\|\leq |n_f(\|x-x_*\|)|<\|x-x_*\| $,  thus leading to the
desired result.
\end{proof}
\begin{lemma} \label{le:cl2}
If {\bf h3} holds and $\|x-x_*\|\leq t<r$, then $ \|G_F(x)-x_*\|\leq
[ |n_f(t)|/t^{p+1}]\,\|x-x_*\|^{p+1}. $
\end{lemma}
\begin{proof}
The inequality is trivial for $x=x_*$. If $0<\|x-x_*\|\leq t<r$,
combining the assumption {\bf h3} and \eqref{eq:def.nf}, we obtain
$|n_f(\|x-x_*\|)|/\|x-x_*\|^{p+1}\leq |n_f(t)|/t^{p+1}$. So, using
Lemma~\ref{l:wdef}, the statement follows.
\end{proof}

\subsubsection{Optimal ball of convergence and uniqueness} \label{sec:unique}
In this section, we  obtain the optimal convergence radius and
the uniqueness of the solution.
\begin{lemma} \label{pr:best}
If  $f(\rho)/(\rho f'(\rho))-1=1$
and $\rho < \kappa$, then  $r=\rho$ is the optimal convergence radius.
\end{lemma}
\begin{proof}
Assume that $f(\rho)/(\rho f'(\rho))-1=1$ and $\rho < \kappa$.
Define the function $ h:(-\kappa,
\,\kappa)\to \mathbb{R}$ by
\begin{equation} \label{eq:dh1}
 h(t)=
      \begin{cases}
       -f(-t), \quad \;\; t\in  (-\kappa, \,0],\\
        f(t), \quad \quad   \; t\in [0, \,\kappa).
      \end{cases}
\end{equation}
It is straightforward to show that $h(0)=0$,  $h'(0)=-1,$
$h'(t)= f'(|t|)$ and
$$
|h'(0)|\left|h'(t)-h'(\tau t)\right| \leq
    f'(|t|)-f'(\tau|t|), \quad\tau \in [0,1], \quad t\in (-\kappa,\, \kappa).
$$
So, for $F=h$, $\banacha=\banachb=\mathbb{R}$,  $\Omega=(-\kappa,\,
\kappa)$ and $x_*=0$ the assumptions of Theorem~\ref{th:nt} are satisfied. Thus, as $\rho<\kappa $,  it  suffices to show that the
Gauss-Newton method applied for solving   $h(t)=0$, with starting
point $x_0=-\rho$, does not converge. As $f(\rho)/(\rho
f'(\rho)-1=1$, the definition of $h$ in \eqref{eq:dh1} yields
$$
x_{1} =-\rho-h(-\rho)/h'(-\rho)= -\rho+f(\rho)/f'(\rho)=[f(\rho)/(\rho f'(\rho))-1]\rho=\rho.
$$
Again, the definition of $h$ in \eqref{eq:dh1}  and the assumption
$f(\rho)/(\rho f'(\rho)-1=1$ lead to
$$
x_{2} =\rho-h(\rho)/h'(\rho)=\rho-f(\rho)/f'(\rho)=-[f(\rho)/(\rho f'(\rho))-1] \rho=-\rho.
$$
Therefore, the Gauss-Newton method for  solving $h(t)=0$, with staring point $x_0=-\rho$, produces the cycle
$$
x_0=-\rho,\quad x_1=\rho, \quad  x_2=-\rho,\; \ldots \;.
$$
As a consequence, it does not converge. Therefore, the lemma is proved.
\end{proof}

\begin{lemma} \label{pr:uniq}
 The point  $x_*$ is the unique zero of $F$ in $B(x_*, \sigma)$.
\end{lemma}
\begin{proof}
 Assume that $y \in B(x_*, \sigma)$  and $F(y)=0$.  Using $F(x_*)=0$ and $F(y)=0$, we have
$$y-x_*=F'(x_*)^{\dagger}[F'(x_*)(y-x_*)-F(y)+F(x_*)].$$
Combining the last equation with properties of the norm and the
definition of $\beta$,  we obtain
$$
\|y-x_*\|\leq\beta\int_{0}^{1}\|F'(x_*)-F'(x_*+u(y-x_*))\|\|y-x_*\|du.$$
Using \eqref{Hyp:MH} with $x=x_*+u(y-x_*)$, $\tau=0$ and some algebraic manipulation, we easily conclude,  from  the last equality, that
$$
\|y-x_*\|\leq  \int_{0}^{1}[f'(u\|y-x_*\|)-f'(0)]\|y-x_*\|du=
f(\|y-x_*\|)+\|y-x_*\|.
$$
Since $0<\|y-x_*\|< \sigma$, i.e., $f(\|y-x_*\|)<0$, the last inequality implies that
$
\|y-x_*\|< \|y-x_*\|,
$
which is a contradiction. Hence,  $y=x_*$.
\end{proof}

\subsection{Gauss-Newton sequence} \label{sec:proof}
In this section,  we will prove the statements in Theorem~\ref{th:nt} involving  the Gauss-Newton sequence $\{x_k\}$. First,
note that the first equation in  \eqref{eq:DNS} together with \eqref{NF} imply that   the sequence $\{x_k\}$  satisfies
\begin{equation} \label{GF}
x_{k+1}=G_F(x_k),\qquad k=0,1,\ldots \,.
\end{equation}
which is indeed an equivalent definition of this sequence.
\begin{corollary}\label{pr:nthe}
The sequence $\{x_k\}$ is well defined,  contained in $B(x_*,r)$
and it converges to the point $x_*$, which is the unique zero of $f$ in
$B(x_*,\sigma)$. Furthermore it holds:
\begin{equation} \label{eq:q2e}
    \lim_{k \to \infty}\left[\|x_{k+1}-x_*\|\big{/}\|x_k-x_*\|\right]=0.
  \end{equation}
If, additionally, {\bf  h3} holds,   the sequences  $\{x_k\}$  and $\{t_k\}$ satisfy
\begin{equation}\label{eq:q333}
\|x_{k+1}-x_*\| \leq \big[t_{k+1}/t_k^{p+1}\big]\,\|x_k-x_*\|^{p+1}\leq \big[t_{1}/t_0^{p+1}\big]\,\|x_k-x_*\|^{p+1}, \qquad k=0,1,\ldots\,.
  \end{equation}
  Consequently, for $k\geq0$,
$$ \|x_{k}-x_*\| \leq\left\{
                             \begin{array}{ll}
                               t_0[t_1/t_0]^k, & \hbox{if \quad p=0;} \\
                               t_0[t_1/t_0]^{((p+1)^k-1)/p}, & \hbox{if \quad p$\neq$0.}
                             \end{array}
                           \right.$$
\end{corollary}
\begin{proof}
Since $x_0\in B(x_*,r)/\{x_*\},$ i.e., $0<\|x_{0}-x_*\|<r$,  and $r\leq \nu$, combining \eqref{GF}, the inclusion in  Lemma~\ref{l:wdef},
Lemma \ref{wdns} and an induction argument, we conclude that  $\{x_k\}$  is well defined and it remains in $B(x_*,r)$.

We will now prove that $\{x_k\}$ converges to $x_*$. Since $\|x_{k}-x_*\|<r\leq \rho$,  for $ k=0,1,\ldots \,$,  we obtain from
\eqref{GF}, Lemma~\ref{l:wdef} and Proposition~\ref{pr:incr102}, that
\begin{equation}\label{eq:conv1}
0\leq\|x_{k+1}-x_*\|=\|G_F(x_k)-x_*\|\leq |n_f(\|x_{k}-x_*\|)|<\|x_{k}-x_*\|,\qquad  k=0,1,\ldots \,.
\end{equation}
So, $\{\|x_{k}-x_*\| \}$ is a bounded and  strictly decreasing
sequence. Therefore $\{\|x_{k}-x_*\| \}$ converges. Let
$\ell_*=\lim_{k\to \infty}\|x_{k}-x_*\|$.
 Since  $\{\|x_{k}-x_*\| \}$ remains in $(0, \,\rho)$ and is strictly decreasing, we have $0\leq \ell_*<\rho$. Thus, taking
the limit in \eqref{eq:conv1} with $t$ converging to $0$ and using  the  continuity of $n_f$ in $[0, \rho)$,
we obtain that  $0\leq \ell_{*}=|n_f(\ell_*)|$. But, if $\ell_* \neq 0$, Proposition~\ref{pr:incr102}  implies $|n_f(\ell_*)|<\ell_*$,
hence $\ell_*=0$. Therefore,  the convergence $x_k \rightarrow x_*$ is proved. The uniqueness was proved in Lemma~\ref{pr:uniq}.

In order to prove the equality in \eqref{eq:q2e}, note that equation \eqref{eq:conv1} implies
$$
\left[\|x_{k+1}-x_*\|\big{/}\|x_{k}-x_*\|\right]\leq \left[|n_f(\|x_{k}-x_*\|)|\big{/}\|x_{k}-x_*\|\right], \qquad k=0,1, \ldots.
$$
Since $\lim_{k\to \infty}\|x_{k}-x_*\|=0$, the  desired inequality follows from the first statement in  Proposition~\ref{pr:incr102}.

Now we will show  \eqref{eq:q333}. First, we will prove by induction  that the sequences  $\{x_k \}$ and  $\{t_k \}$, defined in \eqref{GF} and \eqref{eq:tknk}, respectively,  satisfy
\begin{equation}\label{eq:mjs}
\|x_{k}-x_*\|\leq t_k, \qquad k=0,1, \ldots.
\end{equation}
Due to $t_0=\|x_0-x_*\|$,  the above inequality holds for $k=0$.
Now, assume that $\|x_{k}-x_*\|\leq t_k$.  Using \eqref{GF},
Lemma~\ref{le:cl2}, the induction assumption  and \eqref{eq:tknk}, we
obtain that
$$
\|x_{k+1}-x_*\|=\|G_F(x_k)-x_*\|\leq \frac{|n_f(t_k)|}{t_{k}^{p+1}}\,\|x_{k}-x_*\|^{p+1}\leq |n_f(t_k)|=t_{k+1},
$$
and \eqref{eq:mjs} holds. Therefore, it is easily seen that the
first inequality in \eqref{eq:q333} follows by combining
 \eqref{GF}, \eqref{eq:mjs},  Lemma~\ref{le:cl2}  and \eqref{eq:tknk}. The second inequality in \eqref{eq:q333} is immediate,
due to the fact that the sequence $\{t_{k+1}/t_k^{p+1}\}$ is
strictly decreasing. Finally, for the last part of the corollary, it
is enough to use \eqref{eq:q333} and some simple algebraic manipulations.
\end{proof}

The proof of Theorem~\ref{th:nt} follows from Corollary~\ref{cr:kanttk}, the Lemmas~\ref{pr:best} and \ref{pr:uniq}  and Corollary~\ref{pr:nthe}.

\section{Special Cases} \label{apl}
In this section, we present some special cases of Theorem~\ref{th:nt}.
\subsection{Convergence results  under H\"{o}lder-like and Smale  conditions}
In this section, we  present  a local convergence theorem for the Gauss-Newton method under a H\"{o}lder-like condition, see \cite{F10,huangy2004}.  We also provide a Smale's theorem on the Gauss-Newton method for analytical functions, cf. \cite{S86}.
\begin{theorem}\label{th:HV}
Let $\banacha$ and $\banachb$ be  Hilbert spaces,
 $\Omega\subseteq \banacha$ be an open set and
$F:{\Omega}\to \banachb$ be a continuously differentiable
function such that $F'$  has a closed image in $\Omega$. Let $x_* \in \Omega,$ $R>0$,
$
\beta:=\|F'(x_*)^{\dagger}\|$ and $ \kappa:=\sup \left\{ t\in [0, R): B(x_*, t)\subset\Omega \right\}.
$
Suppose that $F(x_*)=0$,
$F '(x_*)$ is injective
and there exists a constant $K>0$ and $ 0< p \leq 1$ such that
$$
\beta\left\|F'(x)-F'(x_*+\tau(x-x_*))\right\|\leq  K(1-\tau^p) \|x-x_*\|^p, \qquad   x\in B(x_*, \kappa) \quad \tau \in [0,1].
$$
Let $$r=\min \{\kappa, \,[(p+1)/((2p+1) K)]^{1/p}\}.$$ Then, the sequences $\{x_k\}$ and $\{t_k\}$, with starting points $x_0\in B(x_*, r)/\{x_*\}$ and $t_0=\|x_0-x_*\|$, respectively, such that
\begin{equation}\label{1212}
    x_{k+1} ={x_k}-F'(x_k) ^{\dagger}F(x_k), \qquad t_{k+1} =\frac{ K\,p\, t_{k}^{p+1}}{(p+1)[1-K\,t_k^{p}]},\qquad k=0,1,\ldots\,,
\end{equation}
are well defined; $\{t_k\}$  is strictly decreasing, contained in $(0, r)$ and it converges to $0$. Furthermore, $\{x_k\}$ is contained in $B(x_*,r)$, it converges to the point $x_*$, which is the unique zero of $F$ in $B(x_*, [(p+1)/K]^{1/p})$, and there hold:
\begin{equation}\label{12121}
 \|x_{k+1}-x_*\|
\leq \frac{ K\,p}{(p+1)[1- K\,t_k^{p}]}\,\|x_{k}-x_*\|^{p+1}\leq \frac{ K\,p}{(p+1)[1- K\,\|x_0-x_*\|^{p}]}\,\|x_{k}-x_*\|^{p+1},
\end{equation}
 for all $k=0,1,\ldots,$ and
\begin{equation}\label{12122}
\|x_{k}-x_*\| \leq  \left[ \frac{K\,p\,\|x_0-x_*\|^{p}}{(p+1)[1- K\,\|x_0-x_*\|^{p}]}\right]^{[(p+1)^k-1]/p}\,\|x_0-x_*\|, \qquad k=0,1,\ldots .
\end{equation}
Moreover, if  $[(p+1)/((2p+1) K)]^{1/p}<\kappa$, then $r=[(p+1)/((2p+1)K)]^{1/p}$ is the best  possible convergence radius.
\end{theorem}
\begin{proof}
It is immediate to prove that  $F$, $x_*$ and $f:[0, \kappa)\to \mathbb{R}$, defined by
$
f(t)=Kt^{p+1}/(p+1)-t,
$
satisfy the inequality \eqref{Hyp:MH} and the conditions  {\bf h1}, {\bf h2} and  {\bf h3} in Theorem \ref{th:nt}.
In this case, it is easily seen that  $\rho$ and $\nu$, as defined in Theorem \ref{th:nt}, satisfy
$$
\rho=[(p+1)/((2p+1)K)]^{1/p} \leq \nu=[1/ K]^{1/p},
$$
and, as a consequence,  $r=\min \{\kappa,\; [(p+1)/((2p+1)
K)]^{1/p}\}$. Moreover,   $f(\rho)/(\rho f'(\rho))-1=1$,
$f(0)=f([(p+1)/ K]^{1/p})=0$ and  $f(t)<0 $ for all $t \in (0,
[(p+1)/K]^{1/p})$. Therefore, the statements of the theorem follow from
 Theorem~\ref{th:nt}.
\end{proof}
\begin{remark}
For $p=1$ in the previous theorem, we obtain the convergence of the
Gauss-Newton method for  injective-overdetermined systems of
equations under a Lipschitz condition, as obtained in Corollary~19 of \cite{MAX2}.
\end{remark}
In the following numerical example, the results of this section are
illustrated.

\begin{example}
Let $ (a,b) \in \mathbb{R}^2-\{(0,0)\}$. Consider the function $H:
\mathbb{R}\to \mathbb{R}^{2}$  defined by
$$H(x):=(ax^{4/3}-2x,bx^{4/3}+x)^T.$$
It is easy to check that $H(0)=0$, i.e., $x_*=~0,$  $$ H'(x)=\left(
        \begin{array}{c}
          \frac{4}{3}ax^{1/3}-2 \\
          \frac{4}{3}bx^{1/3}+1 \\
        \end{array}
      \right),
 \qquad
H^{\dagger}(x)={9}/\left({16(a^2+b^2)x^{2/3}-24(2a-b)x^{1/3}+45}\right)H'(x)^T,$$
and
$$
\beta=\sqrt{5}/5, \qquad \beta\left\|H'(x)-H'(\tau x)\right\|\leq
(4\sqrt{5(a^2+b^2)}/15)(1-\tau^{1/3}) |x|^{1/3}, \quad   x\in
\mathbb{R} \quad \tau \in [0,1].
$$
Hence, applying the Theorem~\ref{th:HV} with
$$x_*=~0, \quad F=H, \quad p=1/3, \quad K=(4\sqrt{5(a^2+b^2)}/15), \quad r=(3/\sqrt{5(a^2+b^2)})^3,$$
we can conclude that the sequences $\{x_k\}$ and $\{t_k\}$ as
defined in \eqref{1212}, with starting points $x_0\in
B(0,(3/\sqrt{5(a^2+b^2)})^3)/\{0\}$
 and $t_0=\|x_0\|$, respectively, are well defined; $\{t_k\}$  is strictly decreasing,  contained in $(0, (3/\sqrt{5(a^2+b^2)})^3)$
 and it converges to $0$. Furthermore, $\{x_k\}$ is contained in $B(0,(3/\sqrt{5(a^2+b^2)})^3)$, it converges to the point $x_*$,
 which is the unique zero of $F$ in $B(x_*,(5/\sqrt{5(a^2+b^2)})^3)$,  and the  inequalities \eqref{12121}, and \eqref{12122} hold.
 Moreover, $(3/\sqrt{5(a^2+b^2)})^3$ is the best  possible convergence radius.
\end{example}
Below, we present a theorem correspondent to Theorem \ref{th:nt} under  Smale's condition,  which  has  first appeared  in Dedieu and Shub \cite{MR1651750}, see also Corollary~23
of  \cite{MAX2}.
\begin{theorem}\label{theo:Smale}
Let $\banacha$ and $\banachb$ be  Hilbert spaces,
 $\Omega\subseteq \banacha$ be an open set and
$F:{\Omega}\to \banachb$  an analytic function such that $F'$  has a closed image in $\Omega$. Let $x_* \in \Omega,$ $R>0$,
$
\beta:=\|F'(x_*)^{\dagger}\| $ and $ \kappa:=\sup \left\{ t\in [0, R): B(x_*, t)\subset\Omega \right\}.
$
Suppose that $F(x_*)=0$,  $F '(x_*)$ is injective and
$$
 \gamma := \sup _{ n > 1 }\beta\left\| \frac
{F^{(n)}(x_*)}{n !}\right\|^{1/(n-1)}<+\infty.
$$
Let
$$
r:=\min \left\{\kappa,\big( 5-
\sqrt{17}\big)/(4\gamma)\right\}.
$$
Then, the sequences $\{x_k\}$ and $\{t_k\}$, with starting points $x_0\in B(x_*, r)/\{x_*\}$ and $t_0=\|x_0-x_*\|$, respectively, such that
$$
    x_{k+1} ={x_k}-F'(x_k) ^{\dagger}F(x_k), \qquad t_{k+1} =\frac{  \gamma t_{k}^{2}}{2(1-\gamma t_k)^2-1 },\qquad k=0,1,\ldots\,,
$$
are well defined; $\{t_k\}$  is strictly decreasing, contained in $(0, r)$ and it converges to $0$. Furthermore,  $\{x_k\}$ is contained in $B(x_*,r)$,
it converges to the point $x_*$, which is the unique zero of $F$ in $B(x_*,1/(2\gamma))$, and there hold:
$$
    \|x_{k+1}-x_*\| \leq
    \frac{\gamma}{2(1-\gamma t_k)^2-1}\|x_k-x_*\|^2\leq \frac{\gamma}{2(1-\gamma \|x_0-x_*\|)^2-1}\|x_k-x_*\|^2,\qquad  k=0,1,\ldots,
$$
and
$$
\|x_{k}-x_*\| \leq  \left[ \frac{\gamma\|x_0-x_*\|}{2(1- \gamma\,\|x_0-x_*\|)^2-1}\right]^{(2^k-1)}\,\|x_0-x_*\|, \qquad k=0,1,\ldots .
$$
Moreover, if  $(5-
\sqrt{17})/(4\gamma)<\kappa$, then
$r=(5-
\sqrt{17})/(4\gamma)$ is the
best possible convergence  radius.
\end{theorem}
\begin{proof}
In this case, the  real function, $f:[0,1/\gamma) \to \mathbb{R}$,
defined by $ f(t)={t}/{(1-\gamma t)}-2t$, is a majorant function for
the function $F$ on $B(x_*, 1/\gamma)$. Hence, as  $f$ has a convex
derivative, the proof follows the same pattern as outlined in
Theorem~20 of \cite{MAX2}.
\end{proof}

\subsection{Convergence result  under a generalized Lipschitz condition}
In this section, we  present a local convergence theorem for the
Gauss-Newton method  under a generalized Lipschitz condition
according to X.Wang (see \cite{huangy2004,XW10}). It is worth to
point out that the result in this section does not assume that the
function defining the generalized Lipschitz condition  is
nondecreasing.

\begin{theorem} \label{th:XWT}
Let $\banacha$ and $\banachb$ be  Hilbert spaces,
 $\Omega\subseteq \banacha$ be an open set and
$F:{\Omega}\to \banachb$ be a continuously differentiable
function such that $F'$  has a closed image in $\Omega$. Let $x_* \in \Omega,$ $R>0$,
$
\beta:=\|F'(x_*)^{\dagger}\|$ and $ \kappa:=\sup \left\{ t\in [0, R): B(x_*, t)\subset\Omega \right\}.
$
Suppose that $F(x_*)=0$,
$F '(x_*)$ is injective and there exists a  positive  integrable function $L:[0,\; R)\to \mathbb{R}$ such that
\begin{equation}\label{Hyp:XW}
\beta\left\|F'(x)-F'(x_*+\tau(x-x_*))\right\| \leq  \int^{\|x-x_*\|}_{\tau\|x-x_*\|} L(u){\rm d}u,
\end{equation}
for all $\tau \in [0,1]$, $x\in B(x_*, \kappa)$. Let positive constants
$$
\bar{\nu}:=\sup \left\{t\in [0, R): \displaystyle \int_{0}^{t}L(u){\rm d}u-1 < 0\right\},
$$

$$
\bar{\rho}:=\sup \left\{t\in (0, \delta):
\displaystyle \int^{t}_{0}L(u)u {\rm d}u\Big{/}\left[t\left(1-\displaystyle  \int^{t}_{0}L(u){\rm d}u\right)\right]<1, \; t\in (0, \delta)\right\}, \qquad \bar{r}=\min \left\{\kappa, \bar{\rho}\right\}.
$$
Then, the sequences $\{x_k\}$ and $\{t_k\}$, with starting point $x_0\in B(x_*, \bar{r})/\{x_*\}$ and $t_0=\|x_0-x_*\|$, respectively, such that
$$
    x_{k+1} ={x_k}-F'(x_k) ^{\dagger}F(x_k), \qquad t_{k+1} =\displaystyle  \int^{t_k}_{0}L(u)u {\rm d}u\Big{/}\left(1-\displaystyle  \int^{t_k}_{0}L(u){\rm d}u\right),\qquad k=0,1,\ldots\,,
$$
are well defined, $\{t_k\}$ is strictly decreasing,  contained in $(0, \bar{r})$ and it converges to $0$. Furthermore, $\{x_k\}$ is contained in $B(x_*,\bar{r})$, it converges to  $x_*$, which is the unique zero of $F$ in $B(x_*, \bar{\sigma})$, where
$$
\bar{\sigma}:=\sup\left \{t\in(0, \kappa):  \int^{t}_{0}L(u)(t-u){\rm d}u- t< 0 \right\},
$$
and there  hold: $\lim_{k\to \infty}t_{k+1}/t_k=0$ and
$
\lim_{k\to \infty}[\|x_{k+1}-x_*\|/\|x_k-x_*\|]=0.
$
Moreover, if
$$
\displaystyle  \int^{\bar{\rho}}_{0}L(u)u {\rm d}u\Big{/}\left[\bar{\rho}\left(1-\displaystyle  \int^{\bar{\rho}}_{0}L(u){\rm d}u\right)\right]= 1,
$$
 and ${\bar \rho}<\kappa$, then $\bar{r}=\bar{\rho}$ is the best  possible convergence radius.

\noindent
If, additionally, given $0\leq p\leq1$
\begin{itemize}
  \item[{ ${\bf h)}$}]  the function   $(0,\, \nu) \ni t \mapsto t^{1-p}L(t)$ is nondecreasing,
\end{itemize}
  then the sequence $\{t_{k+1}/t_k^{p+1}\}$
  is strictly decreasing and we have:
$$
\|x_{k+1}-x_*\| \leq \big[t_{k+1}/t_k^{p+1}\big]\,\|x_k-x_*\|^{p+1}\leq \big[t_{1}/t_0^{p+1}\big]\,\|x_k-x_*\|^{p+1}, \qquad k=0,1,\ldots\,.
$$
  Consequently, for $k\geq0$,
$$ \|x_{k}-x_*\| \leq\left\{
                             \begin{array}{ll}
                               t_0[t_1/t_0]^k, & \hbox{if \quad p=0;} \\
                               t_0[t_1/t_0]^{((p+1)^k-1)/p}, & \hbox{if \quad p$\neq$0.}
                             \end{array}
                           \right.$$

\end{theorem}
\begin{proof}
Let  ${\bar f}:[0, \kappa)\to \mathbb{R}$ be a differentiable function defined by
\begin{equation} \label{eq:wf}
{\bar f}(t)=\int_{0}^{t}L(u)(t-u){\rm d}u-t.
\end{equation}
Note  that the derivative of the function $f$ is given by
$$
 {\bar f}'(t)=\int_{0}^{t}L(u){\rm d}u-1.
$$
Since $L$ is integrable, ${\bar f}'$ is continuous (in fact ${\bar f}'$ is absolutely continuous). Hence, it is easy to see that \eqref{Hyp:XW} becomes
 \eqref{Hyp:MH} with $f'={\bar f}'$. Moreover, since $L$ is positive, the function $f={\bar f}$  satisfies
the conditions  {\bf h1} and  {\bf h2} in Theorem \ref{th:nt}. Direct algebraic manipulation yields
$$
\frac{1}{t^{p+1}} \left[\frac{{\bar f}(t)}{{\bar f}'(t)}-t\right]=\left[
\frac{1}{t^{p+1}}\displaystyle\int^{t}_{0}L(u)u{\rm d}u\right]
\frac{1}{|{\bar f}'(t)|}.
$$
If assumption {\bf  h} holds, then Lemma~$2.2$ of \cite{XW9} implies
that the first term on the right hand side of the above equation is
nondecreasing
  in $(0,\, \nu)$. Now,  since $1/|{\bar f}'|$ is  strictly  increasing in $(0,\, \nu)$, the above equation implies that {\bf  h3} in Theorem~\ref{th:nt},
  with $f={\bar f}$, also holds.
Therefore, the result  follows from Theorem~\ref{th:nt} with $f={\bar f}$, $\nu=\bar{\nu}$, $\rho=\bar{\rho}$, $r=\bar{r}$ and $\sigma=\bar{\sigma}$.
\end{proof}
\begin{remark}
If the positive integrable function $L:[0,\; R)\to \mathbb{R}$ is  nondecreasing, then  the strictly increasing function $f':[0,\; R)\to \mathbb{R}$, defined by
$$
 f'(t)=\int_{0}^{t}L(u){\rm d}u-1,
$$
is   convex. Hence,  the sequence generated by the Gauss-Newton method
converges with quadratic rate, see for example Corollary~8  of
\cite{MAX2}. Moreover, in this case it is not hard to prove that the
inequalities \eqref{Hyp:MH} and \eqref{Hyp:XW} are equivalent.
However,  if $f'$ is strictly increasing and not necessarily convex,
the inequalities  \eqref{Hyp:MH} and \eqref{Hyp:XW} are not
equivalent, because there exist continuous and strictly increasing functions with derivative zero almost everywhere. These functions
are not absolutely continuous, i.e.,
 they cannot be  represented by an integral, see examples in  \cite{OW07,T78}.
\end{remark}
\section{Final remarks } \label{fr}

The inexact Gauss-Newton like methods for solving  \eqref{eq:11} are described as follows: Given an initial point $x_0 \in {\Omega}$, define
$$
x_{k+1}={x_k}+S_k, \qquad B(x_k)S_k=-F'(x_k)^*F(x_{k})+r_{k}, \qquad
k=0,1,\ldots,
$$
where $B(x_k)$ is a suitable invertible approximation of the
derivative $F'(x_k)^*F'(x_{k})$,  the residual tolerance, $r_k$, and
the preconditioning invertible matrix, $P_{k}$,   are such that
$$
\|P_{k}r_{k}\|\leq \theta_{k}\|P_{k}F'(x_k)^*F(x_{k})\|,
$$
for a suitable forcing number $\theta_{k}$. It would be interesting to
study  this class of  methods under a majorant condition, without the
convexity assumption on the derivative of the majorant function.
This analysis will be carried out in the future.

\bibliographystyle{abbrv}

\begin{thebibliography}{10}


\bibitem{AAA}
J.~Chen.
\newblock The convergence analysis of inexact {G}auss-{N}ewton methods for
  nonlinear problems.
\newblock {\em Comput. Optim. Appl.}, 40(1):97--118, 2008.


\bibitem{C11}
J.~Chen and W.~Li.
\newblock Convergence of {G}auss-{N}ewton's method and uniqueness of the
  solution.
\newblock {\em Appl. Math. Comput.}, 170(1):686--705, 2005.

\bibitem{MR1651750}
J.~P. Dedieu and M.~Shub.
\newblock {N}ewton's method for overdetermined systems of equations.
\newblock {\em Math. Comp.}, 69(231):1099--1115, 2000.

\bibitem{F08}
O.~P. Ferreira.
\newblock Local convergence of {N}ewton's method in {B}anach space from the
  viewpoint of the majorant principle.
\newblock {\em IMA J. Numer. Anal.}, 29(3):746--759, 2009.

\bibitem{F10}
O.~P. Ferreira.
\newblock Local convergence of {N}ewton's method under majorant condition.
\newblock {\em J. Comput. Appl. Math.}, 235(5):1515--1522, 2011.

\bibitem{MAX1}
O.~P. Ferreira and M.~L.~N. Gon{\c{c}}alves.
\newblock Local convergence analysis of inexact {N}ewton-like methods under
  majorant condition.
\newblock {\em Comput. Optim. Appl.}, 48(1):1-21, 2011.



\bibitem{MAX2}
O.~P. Ferreira, M.~L.~N. Gon{\c{c}}alves, and P.~R. Oliveira.
\newblock Local convergence analysis of the {G}auss-{N}ewton method under a
  majorant condition.
\newblock {\em J. Complexity}, 27(1):111--125, 2011.

\bibitem{MAX3}
O.~P. Ferreira, M.~L.~N. Gon{\c{c}}alves, and P.~R. Oliveira.
\newblock Local convergence analysis of inexact {G}auss-{N}ewton like methods
  under majorant condition.
\newblock {\em J. Comput. Appl. Math.}, 236(9):2487--2498, 2012.

\bibitem{FS06}
O.~P. Ferreira and B.~F. Svaiter.
\newblock Kantorovich's majorants principle for {N}ewton's method.
\newblock {\em Comput. Optim. Appl.}, 42(2):213--229, 2009.

\bibitem{huangy2004}
Z.~Huang.
\newblock The convergence ball of {N}ewton's method and the uniqueness ball of
  equations under {H}\"older-type continuous derivatives.
\newblock {\em Comput. Math. Appl.}, 47(2-3):247--251, 2004.

\bibitem{LZJ}
C.~Li, W.~H. Zhang, and X.~Q. Jin.
\newblock Convergence and uniqueness properties of {G}auss-{N}ewton's method.
\newblock {\em Comput. Math. Appl.}, 47(6-7):1057--1067, 2004.

\bibitem{OW07}
H.~Okamoto and M.~Wunsch.
\newblock A geometric construction of continuous, strictly increasing singular
  functions.
\newblock {\em Proc. Japan Acad. Ser. A Math. Sci.}, 83(7):114--118, 2007.

\bibitem{S86}
S.~Smale.
\newblock {N}ewton's method estimates from data at one point.
\newblock In {\em The merging of disciplines: new directions in pure, applied,
  and computational mathematics ({L}aramie, {W}yo., 1985)}, pages 185--196.
  Springer, New York, 1986.

\bibitem{G1}
G.~W. Stewart.
\newblock On the continuity of the generalized inverse.
\newblock {\em SIAM J. Appl. Math.}, 17:33--45, 1969.

\bibitem{T78}
L.~Tak{\'a}cs.
\newblock An increasing continuous singular function.
\newblock {\em Amer. Math. Monthly}, 85(1):35--37, 1978.


\bibitem{XW10}
X.~Wang.
\newblock Convergence of {N}ewton's method and uniqueness of the solution of
  equations in {B}anach space.
\newblock {\em IMA J. Numer. Anal.}, 20(1):123--134, 2000.



\bibitem{XW9}
X.~H. Wang and C.~Li.
\newblock Convergence of {N}ewton's method and uniqueness of the solution of
  equations in {B}anach spaces. {II}.
\newblock {\em Acta Math. Sin. (Engl. Ser.)}, 19(2):405--412, 2003.

\bibitem{W1}
P.~A. Wedin.
\newblock Perturbation theory for pseudo-inverses.
\newblock {\em Nordisk Tidskr. Informationsbehandling (BIT)}, 13:217--232.

\bibitem{10100}
X.~Xu and C.~Li.
\newblock Convergence of {N}ewton's method for  systems of equations with constant rank derivatives.
\newblock {\em J.  Comput. Math.}, 25 (6):705--718, 2007.


\bibitem{1010}
X.~Xu and C.~Li.
\newblock Convergence criterion of {N}ewton's method for singular systems with constant rank derivatives.
\newblock {\em J.  Math. Anal.  Appl.}, 345 (2):689--701, 2008.



\end{thebibliography}


\end{document}